\documentclass{article} 

\usepackage{amsmath,amsthm,amssymb,amsfonts,graphicx}

\usepackage{geometry} 
\newtheorem{Theorem}{Theorem}[section]
\theoremstyle{plain}

\newtheorem{Lemma}[Theorem]{Lemma}

\newtheorem{Remark}[Theorem]{Remark}
\newtheorem{Corollary}[Theorem]{Corollary}
\newtheorem{Proposition}[Theorem]{Proposition}

\numberwithin{equation}{section}

\title{Limiting Distribution of the Rightmost Particle in Catalytic Branching Brownian Motion.}
\author{Sergey Bocharov\footnote{Department of Mathematics, Zhejiang University, China} \ and Simon C. Harris\footnote{Department of Mathematical 
Sciences, Bath University, UK}}

\begin{document}

\maketitle

\begin{abstract}
We study the model of binary branching Brownian motion with spatially-inhomogeneous branching rate 
$\beta \delta_0(\cdot)$, where $\delta_0(\cdot)$ is the Dirac delta function and $\beta$ is some positive constant. We show that 
the distribution of the rightmost particle centred about $\frac{\beta}{2}t$ converges to a mixture of Gumbel distributions according to a martingale limit. 
Our results form a natural extension to S. Lalley and T. Sellke \cite{6} for the degenerate case of catalytic branching.
\end{abstract}

\section{Introduction and Main Results.}

\subsection{Model}
In this article we consider the model of branching Brownian motion with binary splitting and spatially inhomogeneous 
branching rate $\beta \delta_0(\cdot)$, where $\delta_0(\cdot)$ is the Dirac delta function and $\beta > 0$ is some constant.

In such a model we start with a single particle whose path $(X_t)_{t \geq 0}$ is distributed like a standard Brownian motion. Then, 
at a random time $T$ (the branching time) satisfying
\[
\mathbb{P} \Big( T > t \ \big\vert \ (X_s)_{s \geq 0} \Big) = \mathrm{e}^{- \beta L_t} \text{,}
\]
where $(L_t)_{t \geq 0}$ is the local time at $0$ of $(X_t)_{t \geq 0}$, the initial particle dies and is replaced with two new 
particles, which independently repeat the behaviour of their parent (that is, they move as Brownian motions until their split times when new 
particle emerge, etc.)

Informally, we can write $L_t = \int_0^t \delta_0(X_s) \mathrm{d}s$ thus justifying calling the branching rate $\beta \delta_0(\cdot)$. Also, 
the branching in this model can only take place at the origin since $(L_t)_{t \geq 0}$ only grows on the zero set of $(X_t)_{t \geq 0}$ and stays 
constant elsewhere.

\subsection{Main Result}
Before we state the main result of this article (Theorem \ref{main}) let us define the notation and recall some of the existing results for this catalytic model 
in \cite{1}.

Let us denote by $P$ the probability measure associated to the branching process with $E$ the corresponding expectation. 
We denote the set of all the particles in the system at time $t$ by $N_t$. For every particle $u \in N_t$ we denote by $X^u_t$ its 
spatial position at time $t$. Finally, we define 
\[
R_t := \sup_{u \in N_t} X^u_t \qquad , \ t \geq 0
\]
to be the rightmost particle.

Previously in \cite{1} we have shown that 
\begin{equation}
\label{as_limit}
\frac{R_t}{t} \to \frac{\beta}{2} \qquad P \text{-a.s. } \quad \text{ as } t \to \infty
\end{equation}
The aim of this paper is to prove that $R_t - \frac{\beta}{2}t$ converges in distribution to a non-trivial limit and to describe the limiting 
distribution.

Let us recall from \cite{1} that the process 
\[
M_t := \mathrm{e}^{-\frac{\beta^2}{2}t} \sum_{u \in N_t} \mathrm{e}^{- \beta |X^u_t|} \qquad , \ t \geq 0
\]
is a $P$-martingale of mean $1$ that converges almost surely to a strictly positive limit, which we denote by $M_{\infty}$.

We are now in the position to state our main result.
\begin{Theorem}
\label{main}
For a branching process initiated from $x \in \mathbb{R}$ and any $y \in \mathbb{R}$ we have 
\begin{equation}
\label{main_eq}
\lim_{t \to \infty} P^{x} \Big( R_t \leq \frac{\beta}{2}t + y \Big) = E^{x} \exp \Big\{ -M_{\infty} \mathrm{e}^{- \beta y}\Big\}
\end{equation}
\end{Theorem}
The limiting distribution is thus an average over a family of Gumbel distributions with scale parameter $\beta^{-1}$ and random location $- \beta^{-1} \log M_{\infty}$. 

\subsection{Comparison with other branching Brownian motion models}
A similar formula for branching Brownian motion with spatially-homogeneous branching rate $\beta$ was proved by S. Lalley and T. Sellke in \cite{5}. 
Another similar formula for a general class of branching random walks  in discrete time with spatially-homogeneous branching rate was recently obtained by 
M. Bramson, J. Ding and O. Zeitouni in \cite{2}. 
However of particular relevance to our result is the following theorem due to Lalley and Sellke, which covers a certain class of spatially-inhomogeneous 
branching rates $\beta(x)$, not including the degenerate catalytic case $\beta \delta_0(x)$.
\begin{Theorem}[S. Lalley, T. Sellke, \cite{6}]
\label{Lalley_Sellke}
Consider a binary branching Brownian motion with branching rate $\beta(x)$, where $\beta(\cdot)$ is a continuous function 
such that $\beta(x) \to 0$ as $|x| \to \infty$ and $\int_{- \infty}^{\infty} \beta(x) \mathrm{d}x < \infty$. Let $\lambda_0$ be the largest positive 
eigenvalue of the differential operator $\mathcal{L}: g \mapsto \frac{1}{2}g'' + \beta g$ with the corresponding unique eigenfunction $\varphi_0(\cdot)$, 
normalised so that $\varphi_0(0)=1$. 
Then 
\[
\lim_{t \to \infty} P \Big( R_t \leq \sqrt{\frac{\lambda_0}{2}}t + y \Big) = E \exp \Big\{ 
- Z_{\infty} \gamma \mathrm{e}^{- \sqrt{2 \lambda_0} y} \Big\} \text{ ,}
\]
where $Z_{\infty}$ is the almost sure limit of the martingale $Z_t = \mathrm{e}^{- \lambda_0 t} 
\sum_{u \in N_t} \varphi_0(X^u_t)$, $t \geq 0$, \newline $\gamma = \frac{1}{2 \lambda_0} \int_{- \infty}^{\infty} \mathrm{e}^{\sqrt{2 \lambda_0} x} \beta(x) \nu(\mathrm{d}x)$ 
and $\nu(J) = \int_J \varphi_0(x) \mathrm{d}x$.
\end{Theorem}
The proof of Lalley and Sellke of Theorem \ref{Lalley_Sellke} is based on stochastic comparison of the branching process with a Poisson tidal wave 
and involves a coupling argument. Rather than trying to adapt their proof to suit our model we take an alternative and more direct approach which can be summarised as follows. 

In Section 2 we establish a formula for second moments of quantities of the form $\sum_{u \in N_t} f(X^u_t)$, which in itself is an 
interesting and useful result. We then use this formula to give a lower bound on $P(R_t > \frac{\beta}{2}t + y)$ 
via the Paley-Zygmund inequality. The corresponding upper bound trivially follows from the Markov inequality. 

In Section 3, we can then show that if $|x_0(t)|$ is not too large and $z(t)$ goes to infinity not too fast, then 
\begin{equation}
\label{estimate}
P^{x_0(t)} \Big( R_t \leq \frac{\beta}{2}t + z(t) \Big) \approx 1 - \mathrm{e}^{- \beta |x_0(t)| - \beta z(t)} \text{ ,}
\end{equation}
this being made precise in Proposition \ref{prop_main}. We then use \eqref{estimate} in the identity 
\[
P\Big( R_t \leq \frac{\beta}{2}t + y\Big) = E \Big[ \prod_{u \in N_s} P^{X^u_s} \Big( R_{t - s} \leq \frac{\beta}{2}t + y \Big) \Big] 
\]
to get the main result.

\section{First and second moments computations.}
For $\lambda \in \mathbb{R}$ and $t \geq 0$ let us define 
\[
N_t^{\lambda} := \{ u \in N_t \ : \ X^u_t \geq \lambda\}
\]
to  be the set of particles at time $t$ which lie above level $\lambda$. In this section we are going to study the asymptotic properties of the first two moments of $|N_t^{\lambda}|$ for 
$\lambda = \frac{\beta}{2}t + y$.

\subsection{'Many-to-One' lemma and applications}
Let us extend the branching process by introducing an infinite line of descent (a sequence of particles) which we call the spine and which is chosen uniformly 
at random from all the possible lines of descent. More precisely, the spine starts with the initial particle of the branching process. It continues with one 
of the children of the initial particle chosen with probability $\frac{1}{2}$, then with one of the chosen child's child with probability $\frac{1}{2}$ and so on. 

We let $\tilde{P}$ be the extension of the probability measure $P$ so that the branching process under $\tilde{P}$ is defined together 
with the spine as described above. We denote the expectation associated to $\tilde{P}$ by $\tilde{E}$. We also let $\xi_t$ denote the position of the 
particle in the spine at time $t$. It is not hard to see that $(\xi_t)_{t \geq 0}$ is a Brownian Motion under 
$\tilde{P}$. We let $(\tilde{L}_t)_{t \geq 0}$ be the local time at $0$ of $(\xi_t)_{t \geq 0}$.

Recall a special case of the 'Many-to-One' Lemma, as was used extensively in \cite{1} .
\begin{Lemma}['Many-to-One' Lemma]
\label{many_one}
Suppose that $f(\cdot) : \mathbb{R} \to \mathbb{R}$ is a non-negative measurable function. Then 
\[
E \sum_{u \in N_t} f(X^u_t) = \tilde{E} \Big[ f(\xi_t) \mathrm{e}^{\beta \tilde{L}_t} \Big]  \text{.}
\]
\end{Lemma}

Let us also recall a standard result (see e.g. \cite{4}) that if $(X_t)_{t \geq 0}$ is a Brownian motion under $\mathbb{P}$ and $(L_t)_{t \geq 0}$ is its local time at $0$ 
then the joint density of $X_t$ and $L_t$ at any time $t > 0$ is
\begin{equation}
\label{density}
\mathbb{P} \big( X_t \in \mathrm{d}x , L_t \in \mathrm{d}l \big) = \frac{|x| + l}{ \sqrt{2 \pi t^3}} \exp \Big\{ - \frac{(|x| + l)^2}{2 t}
\Big\} \mathrm{d}x \mathrm{d}l \text{ ,} \qquad x \in \mathbb{R} \text{, } l > 0 \text{.}
\end{equation}
Lemma \ref{many_one} together with \eqref{density} yields the following simple formula for $E \big\vert N_t^{\lambda} \big\vert$.
\begin{Proposition}
\label{first}
For $\lambda > 0$
\begin{equation}
\label{E_exact}
E \big\vert N_t^{\lambda} \big\vert = \Phi \big( \beta \sqrt{t} - \frac{y}{\sqrt{t}} \big) \mathrm{e}^{- \frac{\beta^2}{2}t - \beta \lambda } \text{ ,}
\end{equation}
where $\Phi(x) = \frac{1}{\sqrt{2 \pi}} \int_{- \infty}^x \mathrm{e}^{-\frac{y^2}{2}} \mathrm{d}y$ is the cumulative distribution function of a standard normal. In 
particular, for $t$ sufficiently large so that $\frac{\beta}{2}t + y > 0$, 
\begin{equation}
\label{E_exact1}
E \big\vert N_t^{\frac{\beta}{2}t + y} \big\vert = \Phi \big( \frac{\beta}{2} \sqrt{t} - \frac{y}{\sqrt{t}} \big) \mathrm{e}^{- \beta y} \text{.}
\end{equation}
\end{Proposition}
\begin{proof}
Take $f(\cdot) = \mathbf{1}_{[\lambda, \infty)} (\cdot)$ in Lemma \ref{many_one}. Then
\[
E |N_t^{\lambda}| = E \Big[ \sum_{u \in N_t} \mathbf{1}_{\{X^u_t \geq \lambda\}} \Big] = \tilde{E} \Big[ \mathbf{1}_{\{\xi_t \geq \lambda\}} \ \mathrm{e}^{\beta \tilde{L}_t} \Big] \text{.}
\]
Substituting the joint density of $\xi_t$ and $\tilde{L}_t$ from \eqref{density} gives
\begin{align*}
E \big\vert N_t^{\lambda} \big\vert  &= \int_0^{\infty} \int_{\lambda}^{\infty} \mathrm{e}^{\beta l} \frac{x + l}{\sqrt{2 \pi t^3}} 
\exp \Big\{ - \frac{(x + l)^2}{2t} \Big\} \mathrm{d}x \mathrm{d}l \\
&= \int_0^{\infty} \mathrm{e}^{\beta l} \frac{1}{\sqrt{2 \pi t}} \exp \Big\{ - \frac{(\lambda + l)^2}{2t} \Big\} \mathrm{d}l \\
&= \int_0^{\infty} \frac{1}{\sqrt{2 \pi t}} \exp \Big\{ - \frac{1}{2t} \big( l - (\beta t - \lambda)\big)^2 + \frac{\beta^2}{2}t - \beta \lambda \Big\} \mathrm{d}l \\
&= \mathrm{e}^{\frac{\beta^2}{2}t - \beta \lambda} \int_{-(\beta \sqrt{t} - \frac{\lambda}{\sqrt{t}})}^{\infty} \frac{1}{\sqrt{2 \pi}} 
\mathrm{e}^{- \frac{z^2}{2}} \mathrm{d}z \\
&= \Phi \big( \beta \sqrt{t} - \frac{y}{\sqrt{t}} \big) \mathrm{e}^{- \frac{\beta^2}{2}t - \beta \lambda } \text{.}
\end{align*}
\end{proof}
It follows from \eqref{E_exact1} that for any $y \in \mathbb{R}$ and $t > - \frac{2y}{\beta}$
\begin{equation}
\label{ineq_1}
E \big\vert N_t^{\frac{\beta}{2}t + y} \big\vert \leq \mathrm{e}^{- \beta y}
\end{equation}
and for a fixed $y \in \mathbb{R}$
\begin{equation}
\label{lim_1}
E \big\vert N_t^{\frac{\beta}{2}t + y} \big\vert \to \mathrm{e}^{- \beta y} \quad \text{ as } t \to \infty \text{.}
\end{equation}

\subsection{'Many-to-Two' lemma and applications}
The second moment of $|N_t^{\lambda}|$ is harder to deal with. Recently Harris and Roberts \cite{3} established a general 
'Many-to-Few' lemma which allows computing $k$th moments of branching processes in a systematic way.

We shall first state the special case of this formula for binary catalytic branching in Lemma \ref{many_two}. Then we shall convert this 
formula into a more suitable form in Corollary \ref{simpler} and then use this form to get a good estimate of $E\big[|N_t^{\lambda}|^2\big]$.

For this subsection we need to extend the branching process by introducing two independent spines. That is, we have two infinite lines of 
descent started from the initial particle of the branching process which then with probability $\frac{1}{2}$ independently of each other 
choose to follow one of the initial particle's children and so on. We let $\tilde{P}^2$ be the extension of the probability measure $P$ 
under which the branching process is defined with two independent spines.
 
Moreover, we want to define a new probability measure $\tilde{Q}^2$ so that under $\tilde{Q}^2$ the branching process with 
the two spines can be described as follows.
\begin{itemize}
\item We begin with a single particle moving as a Brownian motion and carrying two marks: $1$ and $2$.
\item The particles in the system undergo binary fission and every time a particle branches every mark carried by that 
particle (there could be $0$, $1$ or $2$ such marks) chooses to follow one of the children with probability $\frac{1}{2}$ 
independently of the other mark. Sequences of particles carrying marks $1$ and $2$ thus define two independent spines.
\item The difference from $\tilde{P}^2$ is that under $\tilde{Q}^2$ particles carrying two marks will branch at rate $4 \beta \delta_0(\cdot)$, 
particles carrying one mark will branch at rate $2 \beta \delta_0(\cdot)$ and particles carrying no marks will branch at rate $\beta \delta_0(\cdot)$. 
\end{itemize}

We let $\xi^1_t$ and $\xi^2_t$ be the positions of particles carrying marks $1$ and $2$ respectively so that $(\xi^1_t)_{t \geq 0}$ and 
$(\xi^2_t)_{t \geq 0}$ are two (correlated) Brownian motions. We let $(\tilde{L}^1_t)_{t \geq 0}$ and $(\tilde{L}^2_t)_{t \geq 0}$ 
be the corresponding local times. We also let $T$ be the time when the two marks stop following the same particle (that is, the two spines separate from each other).

In such a setup we have the following special case of a result from \cite{3}:
\begin{Lemma}['Many-to-Two' Lemma]
\label{many_two}
Let $f(\cdot)$, $g(\cdot) : \mathbb{R} \to \mathbb{R}$ be non-negative measurable functions. Then 
\begin{align}
\label{many_to_two_1}
E \Big[ \Big( \sum_{u \in N_t} f(X^u_t) \Big) \Big( \sum_{u \in N_t} g(X^u_t) \Big) \Big] = \ &\tilde{Q}^2 \Big( \mathbf{1}_{\{T > t\}} 
f(\xi_t^1)g(\xi_t^1) \mathrm{e}^{3 \beta \tilde{L}_t^1}\Big) \nonumber\\
+ \ &\tilde{Q}^2 \Big( \mathbf{1}_{\{T \leq t\}} f(\xi_t^1) g(\xi_t^2) \mathrm{e}^{3 \beta \tilde{L}_T^1} \mathrm{e}^{\beta(\tilde{L}_t^1 - \tilde{L}_T^1)} 
\mathrm{e}^{\beta(\tilde{L}_t^2 - \tilde{L}_T^2)} \Big) \text{.} 
\end{align}
\end{Lemma}
To make explicit calculations easier we simplify \eqref{many_to_two_1} in the following form:
\begin{Proposition}
\label{simpler}
Let $f(\cdot)$, $g(\cdot) : \mathbb{R} \to \mathbb{R}$ be non-negative measurable functions and define
\[
S_f(t) := E \Big( \sum_{u \in N_t} f(X^u_t) \Big) 
\]
to be the first moment of $\sum_{u \in N_t} f(X^u_t)$. Then 
\begin{equation}
\label{many_to_two_2}
E \Big[ \Big( \sum_{u \in N_t} f(X^u_t) \Big) \Big( \sum_{u \in N_t} g(X^u_t) \Big) \Big] = S_{fg}(t) + 
2 \int_0^t S_f(t - s) S_g(t - s) \frac{\partial}{\partial s} \Big( 2 \Phi(\beta \sqrt{s}) \mathrm{e}^{\frac{\beta^2}{2}s} \Big) \mathrm{d}s \text{.}
\end{equation}
\end{Proposition}
\begin{proof}[Proof of Proposition \ref{simpler}]
Note that, from the definition of $\tilde{Q}^2$, 
\begin{equation}
\label{Q}
\tilde{Q}^2\Big( T > t \ \big\vert \ (\xi^1_s)_{s \geq 0}\Big) = \mathrm{e}^{-2 \beta \tilde{L}^1_t} \text{.}
\end{equation}
That is, the two spines will split apart at half of the branching rate $4 \beta \delta(\cdot)$. Then the first term of \eqref{many_to_two_1} is just 
\begin{align}
\label{first_term}
\tilde{Q}^2 \Big( \mathbf{1}_{\{T > t\}} f(\xi_t^1)g(\xi_t^1) \mathrm{e}^{3 \beta \tilde{L}_t^1}\Big) &= 
\tilde{Q}^2 \Big( \tilde{Q}^2 \Big( \mathbf{1}_{\{T > t\}} f(\xi_t^1)g(\xi_t^1) \mathrm{e}^{3 \beta \tilde{L}_t^1} \ \big\vert (\xi^1_s)_{s \geq 0} \Big) \Big)\nonumber \\
&= \tilde{Q}^2 \Big( f(\xi_t^1)g(\xi_t^1) \mathrm{e}^{3 \beta \tilde{L}_t^1} \mathrm{e}^{-2 \beta \tilde{L}^1_t}\Big)\nonumber \\ 
&= S_{fg}(t) 
\end{align}
using Lemma \ref{many_one} for the last equality. The second term is more complicated. 

If we let $\hat{\xi}^{1,2}_t := \xi^{1,2}_{T+t} - \xi^{1,2}_T$, $t \geq 0$ and $\hat{L}^{1,2}_t := \tilde{L}^{1,2}_{T+t} - \tilde{L}^{1,2}_T$ then under 
$\tilde{Q}^2$, $(\hat{\xi}_t^{1,2})_{t \geq 0}$ are two independent Brownian motions, both independent of $(\xi^1_t)_{0 \leq t \leq T}$ and 
$(\hat{L}_t^{1,2})_{t \geq 0}$ are their local times. Thus the second term in \eqref{many_to_two_1} is
\begin{align*}
&\tilde{Q}^2 \Big( \mathbf{1}_{\{T \leq t\}} f(\xi_t^1) g(\xi_t^2) \mathrm{e}^{3 \beta \tilde{L}_T^1} \mathrm{e}^{\beta(\tilde{L}_t^1 - \tilde{L}_T^1)} 
\mathrm{e}^{\beta(\tilde{L}_t^2 - \tilde{L}_T^2)} \Big) \\
= &\tilde{Q}^2 \Big( \mathbf{1}_{\{T \leq t\}} f(\hat{\xi}^1_{t-T}) g(\hat{\xi}^2_{t-T}) \mathrm{e}^{3 \beta \tilde{L}_T^1} \mathrm{e}^{\beta \hat{L}^1_{t - T}} 
\mathrm{e}^{\beta \hat{L}^2_{t - T}}\Big) \\
= &\tilde{Q}^2 \Big( \tilde{Q}^2 \Big( \mathbf{1}_{\{T \leq t\}} f(\hat{\xi}^1_{t-T}) g(\hat{\xi}^2_{t-T}) \mathrm{e}^{3 \beta \tilde{L}_T^1} \mathrm{e}^{\beta \hat{L}^1_{t - T}} 
\mathrm{e}^{\beta \hat{L}^2_{t - T}} \ \big\vert \ T , (\xi_t^1)_{0 \leq t \leq T} \Big) \Big)\\
= &\tilde{Q}^2 \Big( \mathbf{1}_{\{T \leq t\}} \mathrm{e}^{3 \beta \tilde{L}_T^1} S_f(t - T) S_g(t - T)\Big)
\end{align*}
using Lemma \ref{many_one} and independence of $(\hat{\xi}^{1}_t)_{t \geq 0} $ and $(\hat{\xi}^{2}_t)_{t \geq 0}$ of each other and of $(\xi_t^1)_{0 \leq t \leq T}$. Then
\begin{align*}
&\tilde{Q}^2 \Big( \mathbf{1}_{\{T \leq t\}} \mathrm{e}^{3 \beta \tilde{L}_T^1} S_f(t - T) S_g(t - T)\Big)\\
= &\tilde{Q}^2 \Big( \tilde{Q}^2 \Big( \mathbf{1}_{\{T \leq t\}} \mathrm{e}^{3 \beta \tilde{L}_T^1} S_f(t - T) S_g(t - T) \ \big\vert \ (\xi^1_s)_{s \geq 0}\Big) \Big)\\
= &\tilde{Q}^2 \Big( \int_0^t \mathrm{e}^{3 \beta \tilde{L}_s^1} S_f(t - s) S_g(t - s) \ \mathrm{d} \big( - \mathrm{e}^{-2 \beta \tilde{L}^1_s} \big) \Big)
\end{align*}
using \eqref{Q}. Then
\begin{align*}
&\tilde{Q}^2 \Big( \int_0^t \mathrm{e}^{3 \beta \tilde{L}_s^1} S_f(t - s) S_g(t - s) \ \mathrm{d} \big( - \mathrm{e}^{-2 \beta \tilde{L}^1_s} \big) \Big)\\
= &\tilde{Q}^2 \Big( \int_0^t \mathrm{e}^{3 \beta \tilde{L}_s^1} S_f(t - s) S_g(t - s) 2 \beta \mathrm{e}^{-2 \beta \tilde{L}^1_s} \mathrm{d} \tilde{L}^1_s \Big)\\
= &2 \tilde{Q}^2 \Big( \int_0^t S_f(t - s) S_g(t - s) \ \mathrm{d} \big( \mathrm{e}^{\beta \tilde{L}^1_s} \big) \Big) \text{.}
\end{align*}
Finally, using integration-by-parts and Fubini's theorem we get
\begin{align}
\label{second_term}
&2 \tilde{Q}^2 \Big( \int_0^t S_f(t - s) S_g(t - s) \ \mathrm{d} \big( \mathrm{e}^{\beta \tilde{L}^1_s} \big) \Big)\nonumber \\
= &2 \tilde{Q}^2 \Big( f(0)g(0) \mathrm{e}^{\beta \tilde{L}^1_t} - S_f(t)S_g(t) - \int_0^t \frac{\partial}{\partial s} \big( 
S_f(t - s) S_g(t - s) \big) \ \mathrm{e}^{\beta \tilde{L}^1_s} \mathrm{d}s \Big)\nonumber \\
= &2 \Big( f(0)g(0) \tilde{Q}^2 \big( \mathrm{e}^{\beta \tilde{L}^1_t} \big) -S_f(t)S_g(t) - \int_0^t \frac{\partial}{\partial s} \big( 
S_f(t - s) S_g(t - s) \big) \ \tilde{Q}^2  \big( \mathrm{e}^{\beta \tilde{L}^1_s} \big) \mathrm{d}s \Big)\nonumber \\
= &2 \int_0^t S_f(t - s) S_g(t - s) \frac{\partial}{\partial s} \Big( \tilde{Q}^2  \big( \mathrm{e}^{\beta \tilde{L}^1_s} \big) \Big) \mathrm{d}s\nonumber \\
= &2 \int_0^t S_f(t - s) S_g(t - s) \frac{\partial}{\partial s} \Big( 2 \Phi(\beta \sqrt{s}) \mathrm{e}^{\frac{\beta^2}{2}s} \Big) \mathrm{d}s \text{,}
\end{align}
which together with \eqref{first_term} gives the sought formula \eqref{many_to_two_2}.
\end{proof}
As a simple application of \eqref{many_to_two_2} we get the following useful inequality.
\begin{Proposition}
\label{second_moment}
For all $y \in \mathbb{R}$, $t > - \frac{2y}{\beta}$
\begin{equation}
\label{sec_mom_ineq}
E \Big[ \big\vert N_t^{\frac{\beta}{2}t + y} \big\vert^2 \Big] \leq \mathrm{e}^{- \beta y} + C \mathrm{e}^{-2 \beta y} \text{,}
\end{equation}
\label{sec_mom_lim}
where $C>0$ is some positive finite constant which doesn't depend on $t$ or $y$. 
\end{Proposition}
\begin{Remark}
\label{rem}
One can also show that for a fixed $y \in \mathbb{R}$
\begin{equation}
\label{lim_const}
\lim_{t \to \infty} E \Big[ \big\vert N_t^{\frac{\beta}{2}t + y} \big\vert^2 \Big] = \mathrm{e}^{- \beta y} + 2(1 + \sqrt{2}) \mathrm{e}^{-2 \beta y} \text{,}
\end{equation}
but it's not so important for this paper.
\end{Remark}

\begin{proof}[Proof of Proposition \ref{second_moment}]
We substitute $f(\cdot) = g(\cdot) = \mathbf{1}_{[\frac{\beta}{2}t + y, \infty)}(\cdot)$ in \eqref{many_to_two_2} so that 
\[
E \Big[ \big\vert N_t^{\frac{\beta}{2}t + y} \big\vert^2 \Big] = E \big\vert N_t^{\frac{\beta}{2}t + y} \big\vert + 2 \int_0^t 
\Big[ E \big\vert N_{t-s}^{\frac{\beta}{2}t + y} \big\vert \Big]^2 \frac{\partial}{\partial s} \Big( 2 \Phi(\beta \sqrt{s}) \mathrm{e}^{\frac{\beta^2}{2}s} \Big) \mathrm{d}s \text{.}
\]
From \eqref{ineq_1} we know that for $t > -\frac{2y}{\beta}$ and $s \in [0, t]$
\[
E \big\vert N_t^{\frac{\beta}{2}t + y} \big\vert \leq \mathrm{e}^{- \beta y} \text{ and } 
E \big\vert N_{t-s}^{\frac{\beta}{2}t + y} \big\vert = E \big\vert N_{t-s}^{\frac{\beta}{2}(t - s) + \frac{\beta}{2}s + y} \big\vert \leq \mathrm{e}^{- \beta (\frac{\beta}{2}s + y)} \text{.}
\]
Thus, noting that $\frac{\partial}{\partial s} \big( 2 \Phi(\beta \sqrt{s}) \mathrm{e}^{\frac{\beta^2}{2}s} \big) > 0$ since 
$2 \Phi(\beta \sqrt{s}) \mathrm{e}^{\frac{\beta^2}{2}s}$ is increasing in $s$, we get
\begin{align*}
E \Big[ \big\vert N_t^{\frac{\beta}{2}t + y} \big\vert^2 \Big] &\leq \mathrm{e}^{- \beta y} + \mathrm{e}^{- 2 \beta y} 2 \int_0^t 
\mathrm{e}^{- \beta^2 s} \frac{\partial}{\partial s} \Big( 2 \Phi(\beta \sqrt{s}) \mathrm{e}^{\frac{\beta^2}{2}s} \Big) \mathrm{d}s \\
&\leq \mathrm{e}^{- \beta y} + C \mathrm{e}^{-2 \beta y}   \qquad \forall t > -\frac{2y}{\beta} \text{,}
\end{align*}
where $C = 2 \int_0^{\infty} \mathrm{e}^{- \beta^2 s} \frac{\partial}{\partial s} \big( 2 \Phi(\beta \sqrt{s}) \mathrm{e}^{\frac{\beta^2}{2}s} \big) \mathrm{d}s < \infty$.
\end{proof}
Noting that $\{ |N_t^{\frac{\beta}{2}t + y}| > 0\} = \{ R_t > \frac{\beta}{2}t + y\}$ we establish the following simple corollary of Proposition \ref{first} and 
Proposition \ref{second_moment}. 
\begin{Corollary}
\label{rightmost_inequality}
For all $y \in \mathbb{R}$, $t > - \frac{2y}{\beta}$
\[
\mathrm{e}^{- \beta y} (1 - C \mathrm{e}^{- \beta y} ) \Phi \Big(\frac{\beta}{2} \sqrt{t} - \frac{y}{\sqrt{t}} \Big)^2 < P \Big(R_t > \frac{\beta}{2}t + y \Big) < \mathrm{e}^{- \beta y} \text{.}
\] 
\end{Corollary}
So, in particular, $\liminf_{t \to \infty} P \Big(R_t > \frac{\beta}{2}t + y \Big)$, $\limsup_{t \to \infty} P \Big(R_t > \frac{\beta}{2}t + y \Big) \sim 
\mathrm{e}^{- \beta y}$ as $y \to \infty$. 
\begin{proof}[Proof of Corollary \ref{rightmost_inequality}]
Paley-Zygmund and Markov inequality give
\[
\frac{ \Big[ E \big\vert N_t^{ \frac{ \beta}{2}t + y} \big\vert \Big]^2}{E \Big[ \big\vert N_t^{ \frac{ \beta}{2}t + y} \big\vert^2 \Big]} 
\leq P \Big( \big\vert N_t^{\frac{\beta}{2}t + y} \big\vert > 0 \Big) \leq E \big\vert N_t^{\frac{\beta}{2}t + y} \big\vert
\]
Thus applying \eqref{E_exact1} and \eqref{sec_mom_ineq} to the lower bound and \eqref{ineq_1} to the upper bound gives us
\[
\frac{\Phi \Big(\frac{\beta}{2} \sqrt{t} - \frac{y}{\sqrt{t}} \Big)^2 \mathrm{e}^{- 2 \beta y}}{\mathrm{e}^{- \beta y} + C \mathrm{e}^{-2 \beta y}} \leq 
P \Big( \big\vert N_t^{\frac{\beta}{2}t + y} \big\vert > 0 \Big) \leq \mathrm{e}^{- \beta y} \text{.}
\]
Hence
\begin{align*}
\frac{\Phi \Big(\frac{\beta}{2} \sqrt{t} - \frac{y}{\sqrt{t}} \Big)^2 \mathrm{e}^{- 2 \beta y}}{\mathrm{e}^{- \beta y} + C \mathrm{e}^{-2 \beta y}} &= 
\mathrm{e}^{- \beta y} \Big( \frac{1}{1 + C \mathrm{e}^{- \beta y}} \Big) \Phi \Big(\frac{\beta}{2} \sqrt{t} - \frac{y}{\sqrt{t}} \Big)^2\\
&\geq \mathrm{e}^{- \beta y} \big( 1 - C \mathrm{e}^{- \beta y} \big) \Phi \Big(\frac{\beta}{2} \sqrt{t} - \frac{y}{\sqrt{t}} \Big)^2 \text{,}
\end{align*}
which gives the desired inequality.
\end{proof}

\section{Limiting Distribution of the Rightmost Particle.}
This section contains the proof of Theorem \ref{main}. An important preliminary step of the proof is to establish the following consequence of 
Corollary \ref{rightmost_inequality}.
\begin{Proposition} 
\label{prop_main}
Let $x_0(t)$ and $z(t)$ be such that $|x_0(t)| < \frac{1}{4 \beta} \log t$ for $t$ sufficiently large, 
$z(t) \to \infty$ and $z(t) = o(\log t)$ (that is, $\frac{z(t)}{\log t} \to 0$) as $t \to \infty$. Then for $t$ sufficiently large 
\begin{equation}
\label{ineq_main}
1 - \theta_1(t) \mathrm{e}^{- \beta |x_0(t)| - \beta z(t)} \leq P^{x_0(t)} \Big( R_t \leq \frac{\beta}{2}t + z(t) \Big) 
\leq 1 - \theta_2(t) \mathrm{e}^{- \beta |x_0(t)| - \beta z(t)} 
\end{equation}
for some functions $\theta_1(\cdot)$, $\theta_2(\cdot)$ such that $\theta_1(t)$, $\theta_2(t) \to 1$ as $t \to \infty$. 
\end{Proposition}
\begin{proof}
Let $T_0$ be the first time the initial particle of the branching process (started from $x_0$) hits the origin.We fix $\alpha \in (0, \frac{1}{2})$ and write 
\begin{align}
\label{alpha}
P^{x_0(t)} \Big( R_t \leq \frac{\beta}{2}t + z(t) \Big) = &P^{x_0(t)} \Big( R_t \leq \frac{\beta}{2}t + z(t) , \ T_0 \leq \alpha t \Big) \nonumber \\ 
+ &P^{x_0(t)} \Big( R_t \leq \frac{\beta}{2}t + z(t) , \ T_0 > \alpha t \Big)
\end{align}
(the choice of $\alpha$ will become clear later in the proof). Then the first term of \eqref{alpha} can be written as 
\[
P^{x_0(t)} \Big( R_t \leq \frac{\beta}{2}t + z(t) , \ T_0 \leq \alpha t \Big) = 
P^{x_0(t)} \Big( \tilde{R}_{t - T_0} \leq \frac{\beta}{2}t + z(t) , \ T_0 \leq \alpha t \Big) \text{,}
\]
where $\tilde{R}_t := R_{t + T_0}$, $t \geq 0$ is the position of the rightmost particle of the subtree of the original branching process 
started from the origin at time $T_0$. Then conditioning on $T_0$ and using the strong Markov property we get 
\begin{align}
\label{integral}
P^{x_0(t)} \Big( \tilde{R}_{t - T_0} \leq \frac{\beta}{2}t + z(t) , \ T_0 \leq \alpha t \Big) = 
&E^{x_0(t)} \Big[ E^{x_0(t)} \Big( \mathbf{1}_{\{ \tilde{R}_{t - T_0} \leq \frac{\beta}{2}t + z(t) \}} 
\mathbf{1}_{T_0 \leq \alpha t} \ \big\vert \ T_0 \Big) \Big] \nonumber \\
= &\int_0^{\alpha t} P \Big( R_{t-s} \leq \frac{\beta}{2}t + z(t)\Big) f_{T_0}(s) \mathrm{d}s \text{ ,} 
\end{align}
where $f_{T_0}(s) = \frac{|x_0|}{\sqrt{2 \pi s^3}} \mathrm{e}^{- \frac{x_0^2}{2s}}$ is the probability density of $T_0$.

\underline{Lower bound.}

We first prove the lower bound of \eqref{ineq_main}. From \eqref{alpha} and \eqref{integral} we have 
\[
P^{x_0(t)} \Big( R_t \leq \frac{\beta}{2}t + z(t) \Big) \geq \int_0^{\alpha t} P \Big( R_{t-s} \leq \frac{\beta}{2}t + z(t)\Big) f_{T_0}(s) \mathrm{d}s \text{.}
\]
Then from Corollary \ref{rightmost_inequality} we know that for all $t$ sufficiently large (so that $t + \frac{2}{\beta} z(t) > 0$)
\begin{align*}
\int_0^{\alpha t} P \Big( R_{t-s} \leq \frac{\beta}{2}t + z(t)\Big) f_{T_0}(s) \mathrm{d}s \geq 
&\int_0^{\alpha t} \Big( 1 - \exp \big\{ - \beta \big( \frac{\beta}{2}s + z(t)\big) \big\} \Big) f_{T_0}(s) \mathrm{d}s \\
= &P^{x_0(t)}(T_0 \leq \alpha t) - \mathrm{e}^{- z(t)} E^{x_0(t)} \Big( \mathrm{e}^{- \frac{\beta^2}{2} T_0} 
\mathbf{1}_{\{T_0 \leq \alpha t\}} \Big) \\
\geq &P^{x_0(t)}(T_0 \leq \alpha t) - \mathrm{e}^{- z(t)} E^{x_0(t)} \Big( \mathrm{e}^{- \frac{\beta^2}{2} T_0} \Big) \\
= &1 - \mathrm{e}^{- \beta z(t) - \beta |x_0(t)|} - P^{x_0(t)}(T_0 > \alpha t)
\end{align*}
using the fact that $E^{x_0(t)} \mathrm{e}^{- \frac{\beta^2}{2} T_0} = \mathrm{e}^{- \beta |x_0|}$. Then since 
$P^{x_0(t)}(T_0 > \alpha t) \leq \frac{\sqrt{2}|x_0(t)|}{\sqrt{\pi \alpha t}}$ it follows that 
\[
P^{x_0(t)} \Big( R_t \leq \frac{\beta}{2}t + z(t) \Big) \geq 1 - \theta_1(t) \mathrm{e}^{- \beta |x_0(t)| - \beta z(t)} \text{,}
\]
where $\theta_1(t) = 1 + P^{x_0(t)}(T_0 > \alpha t) \mathrm{e}^{ \beta |x_0(t)| + \beta z(t)} \to 1$ as $t \to \infty$ due to assumption that 
$|x_0(t)| < \frac{1}{4 \beta} \log t$ for large enough $t$ and $z(t) = o(\log t)$ as $t \to \infty$.

\underline{Upper bound.}

The upper bound of \ref{ineq_main} is proved similarly. From \eqref{alpha} and \eqref{integral} we have 
\[
P^{x_0(t)} \Big( R_t \leq \frac{\beta}{2}t + z(t) \Big) \leq \int_0^{\alpha t} P \Big( R_{t-s} \leq \frac{\beta}{2}t + z(t)\Big) f_{T_0}(s) \mathrm{d}s + 
P^{x_0(t)} \Big( T_0 > \alpha t \Big) \text{.}
\]
From Corollary \ref{rightmost_inequality} we know that for all $t$ sufficiently large (so that $t + \frac{2}{\beta} z(t) > 0$) 
\begin{align*}
\int_0^{\alpha t} P \Big( R_{t-s} \leq \frac{\beta}{2}t + z(t)\Big) f_{T_0}(s) \mathrm{d}s 
\leq &\int_0^{\alpha t} \Big[ 1 - \exp \big\{ - \frac{\beta^2}{2}s - \beta z(t) \big\} \Big( 1 - \\ 
C \exp \big\{&- \frac{\beta^2}{2}s - \beta z(t) \big\} \Big) \Phi \Big( \frac{\beta}{2} \sqrt{t - s} - \frac{\frac{\beta}{2}s + z(t)}{\sqrt{t - s}} \Big)^2 \Big] f_{T_0}(s) \mathrm{d}s \\
\leq &\int_0^{\alpha t} \Big[ 1 - \exp \big\{ - \frac{\beta^2}{2}s - \beta z(t) \big\} \Big( 1 - \\ 
C \exp \big\{- \beta z(t) \big\} \Big) &\Phi \Big( \frac{\beta}{2} \sqrt{1 - \alpha} \sqrt{t} - \frac{\beta \alpha}{2 \sqrt{1- \alpha} } \sqrt{t}
- \frac{z(t)}{\sqrt{1 - \alpha} \sqrt{t}} \Big)^2 \Big] f_{T_0}(s) \mathrm{d}s \\
\leq &1 - \hat{\theta}_2(t) \exp \{- \beta z(t) \} \int_0^{\alpha t} \mathrm{e}^{- \frac{\beta^2}{2}s} f_{T_0}(s) \mathrm{d}s \text{ ,}
\end{align*}
where 
\[
\hat{\theta}_2(t) = \Big( 1 - C \exp \big\{- \beta z(t) \big\} \Big) \Phi \Big( \frac{\beta}{2} \sqrt{1 - \alpha} \sqrt{t} - \frac{\beta \alpha}{2 \sqrt{1- \alpha} } \sqrt{t}
- \frac{z(t)}{\sqrt{1 - \alpha} \sqrt{t}} \Big)^2 \to 1
\]
as $t \to \infty$ due to the choice of $\alpha < \frac{1}{2}$. Thus 
\begin{align*}
P^{x_0(t)} \Big( R_t \leq \frac{\beta}{2}t + z(t) \Big) \leq &1 - \hat{\theta}_2(t) \exp \{- \beta z(t) \} \int_0^{\alpha t} \mathrm{e}^{- \frac{\beta^2}{2}s} f_{T_0}(s) \mathrm{d}s 
+ P^{x_0(t)} \Big( T_0 > \alpha t \Big)\\
= &1 - \hat{\theta}_2(t) \exp \{- \beta z(t) \} \Big[ E^{x_0(t)} \mathrm{e}^{- \frac{\beta^2}{2}T_0} - 
E^{x_0(t)} \Big( \mathbf{1}_{ \{ T_0 > \alpha t \} } \mathrm{e}^{- \frac{\beta^2}{2}T_0} \Big) \Big] \\ 
&\qquad\qquad\qquad\qquad + P^{x_0(t)} \Big( T_0 > \alpha t \Big)\\
\leq &1 - \hat{\theta}_2(t) \exp \{- \beta z(t) \} \Big[ E^{x_0(t)} \mathrm{e}^{- \frac{\beta^2}{2}T_0} - 
P^{x_0(t)} \Big( T_0 > \alpha t \Big) \Big] \\ 
&\qquad\qquad\qquad\qquad + P^{x_0(t)} \Big( T_0 > \alpha t \Big)\\
= &1 - \theta_2(t) \mathrm{e}^{- \beta |x_0(t)| - \beta z(t)} \text{ ,}
\end{align*}
where 
\[
\theta_2(t) = \hat{\theta}_2(t) \Big( 1 - \mathrm{e}^{\beta |x_0(t)|} P^{x_0(t)} \big( T_0 > \alpha t \big) \Big) - 
\mathrm{e}^{\beta |x_0(t)| + \beta z(t)} P^{x_0(t)} \big( T_0 > \alpha t \big) \to 1
\]
as $t \to \infty$ since $P^{x_0(t)}(T_0 > \alpha t) \leq \frac{\sqrt{2}|x_0(t)|}{\sqrt{\pi \alpha t}}$, $|x_0(t)| < \frac{1}{4 \beta} \log t$ for large enough $t$ 
and $z(t) = o(\log t)$ as $t \to \infty$ and this completes the proof of Proposition \ref{prop_main}.
\end{proof}

Let us now prove the main result of this paper.
\begin{proof}[Proof of Theorem \ref{main}]

Note that for any $t > 0$ and $s < t$ using the Markov property we can write
\begin{align*}
P \Big( R_t \leq \frac{\beta}{2}t + y \Big) &= E \Big( \prod_{u \in N_s} P^{X^u_s} \big( R_{t - s} \leq \frac{\beta}{2}t + y \big) \Big)\\
&= E \Big( \prod_{u \in N_s} P^{X^u_s} \big( R_{t - s} \leq \frac{\beta}{2}(t - s) + \frac{\beta}{2}s + y \big) \Big) \text{.}
\end{align*}
We take $s(t) = \sqrt{\log t}$ in the above formula so that conditions of Proposition \ref{prop_main} will apply. Then fixing an arbitrary $\epsilon > 0$ we write 
\begin{align*}
&\prod_{u \in N_{s(t)}} P^{X^u_{s(t)}} \Big( R_{t - s(t)} \leq \frac{\beta}{2}(t - s(t)) + \frac{\beta}{2}s(t) + y \Big) \\
= &\mathbf{1}_{\{ R_{s(t)} \leq (\frac{\beta}{2} + \epsilon)s(t)\}} \prod_{u \in N_{s(t)}} P^{X^u_{s(t)}} \Big( R_{t - s(t)} \leq \frac{\beta}{2}(t - s(t)) + \frac{\beta}{2}s(t) + y \Big) \\
+ &\mathbf{1}_{\{ R_{s(t)} > (\frac{\beta}{2} + \epsilon)s(t)\}} \prod_{u \in N_{s(t)}} P^{X^u_{s(t)}} \Big( R_{t - s(t)} \leq \frac{\beta}{2}(t - s(t)) + \frac{\beta}{2}s(t) + y \Big) \\
\end{align*}
and as we know from \eqref{as_limit}, $\mathbf{1}_{\{ R_{s(t)} > (\frac{\beta}{2} + \epsilon)s(t)\}} \to 0$
almost surely and hence also
\[
\mathbf{1}_{\{ R_{s(t)} > (\frac{\beta}{2} + \epsilon)s(t)\}} \prod_{u \in N_{s(t)}} P^{X^u_{s(t)}} \Big( R_{t - s(t)} \leq \frac{\beta}{2}(t - s(t)) + \frac{\beta}{2}s(t) + y \Big) 
\leq \mathbf{1}_{\{ R_{s(t)} > (\frac{\beta}{2} + \epsilon)s(t)\}} \to 0
\]
almost surely as $t \to \infty$. On the other hand, we would like to show that on the event $\{ R_{s(t)} \leq (\frac{\beta}{2} + \epsilon)s(t)\}$ 
\begin{equation}
\label{aim}
\prod_{u \in N_{s(t)}} P^{X^u_{s(t)}} \Big( R_{t - s(t)} \leq \frac{\beta}{2}(t - s(t)) + \frac{\beta}{2}s(t) + y \Big) \to \exp \{ - \mathrm{e}^{- \beta y} M_{\infty} \}
\end{equation}
\underline{Upper bound of \eqref{aim}.}
On the event $\{ R_{s(t)} \leq (\frac{\beta}{2} + \epsilon)s(t)\}$ we have 
\begin{align*}
&\prod_{u \in N_{s(t)}} P^{X^u_{s(t)}} \Big( R_{t - s(t)} \leq \frac{\beta}{2}(t - s(t)) + \frac{\beta}{2}s(t) + y \Big) \\
= &\exp \Big\{ \sum_{u \in N_{s(t)}} \log P^{X^u_{s(t)}} \Big( R_{t - s(t)} \leq \frac{\beta}{2}(t - s(t)) + \frac{\beta}{2}s(t) + y \Big) \Big\} \\
\leq &\exp \Big\{ \sum_{u \in N_{s(t)}} \log \Big( 1 - \theta_2\ \big(t - s(t) \big) \mathrm{e}^{- \beta |X^u_{s(t)}| - \frac{\beta^2}{2}s(t) - \beta y} \Big) \Big\}
\end{align*}
for $t$ large enough and where $\theta_2(t) \to 1$ as $t \to \infty$ according to Proposition \ref{prop_main}. Then since $\log(1 - x) \leq -x$ for all 
$x \in \mathbb{R}$ we get (on the event $\{ R_{s(t)} \leq (\frac{\beta}{2} + \epsilon)s(t)\}$)
\begin{align*}
&\prod_{u \in N_{s(t)}} P^{X^u_{s(t)}} \Big( R_{t - s(t)} \leq \frac{\beta}{2}(t - s(t)) + \frac{\beta}{2}s(t) + y \Big) \\
\leq &\exp \Big\{ - \theta_2\big(t - s(t)\big) \mathrm{e}^{- \beta y} \sum_{u \in N_{s(t)}} \mathrm{e}^{- \beta |X^u_{s(t)}| - \frac{\beta^2}{2}s(t)}  \Big\} \\
= &\exp \Big\{ - \theta_2\big(t - s(t)\big) \mathrm{e}^{- \beta y} M_{s(t)} \Big\} \\
\to &\exp \Big\{ - \mathrm{e}^{- \beta y} M_{\infty}\Big\} \text{.}
\end{align*}
\underline{Lower bound of \eqref{aim}.}  
Similarly, on the event $\{ R_{s(t)} \leq (\frac{\beta}{2} + \epsilon)s(t)\}$ we have 
\begin{align*}
&\prod_{u \in N_{s(t)}} P^{X^u_{s(t)}} \Big( R_{t - s(t)} \leq \frac{\beta}{2}(t - s(t)) + \frac{\beta}{2}s(t) + y \Big) \\
\geq &\exp \Big\{ \sum_{u \in N_{s(t)}} \log \Big( 1 - \theta_1\ \big(t - s(t) \big) \mathrm{e}^{- \beta |X^u_{s(t)}| - \frac{\beta^2}{2}s(t) - \beta y} \Big) \Big\}
\end{align*}
for $t$ large enough and where $\theta_1(t) \to 1$ as $t \to \infty$. Then since $\log (1 - x) \geq \frac{\log(1 - x_\ast)}{x_\ast} x$ for all $x \in [0, x_\ast]$ we get 
by taking $x = \theta_1\ \big(t - s(t) \big) \mathrm{e}^{- \beta |X^u_{s(t)}| - \frac{\beta^2}{2}s(t) - \beta y}$ and 
$x_\ast = \theta_1\ \big(t - s(t) \big) \mathrm{e}^{ - \frac{\beta^2}{2}s(t) - \beta y}$ that 
\begin{align*}
&\prod_{u \in N_{s(t)}} P^{X^u_{s(t)}} \Big( R_{t - s(t)} \leq \frac{\beta}{2}(t - s(t)) + \frac{\beta}{2}s(t) + y \Big) \\
\geq &\exp \Big\{ \frac{\log \Big( 1 - \theta_1 \big( t - s(t) \big) \mathrm{e}^{- \frac{\beta^2}{2}s(t) - \beta y} \Big)}{\theta_1 \big( t - s(t) \big) 
\mathrm{e}^{- \frac{\beta^2}{2}s(t) - \beta y}} \theta_1\big(t - s(t)\big) \mathrm{e}^{- \beta y} \sum_{u \in N_{s(t)}} \mathrm{e}^{- \beta |X^u_{s(t)}| - \frac{\beta^2}{2}s(t)}  \Big\} \\
\to &\exp \Big\{ - \mathrm{e}^{- \beta y} M_{\infty}\Big\} 
\end{align*}
using L'Hopitale's rule when taking the limit of the fraction. So we have proved that 
\[
\mathbf{1}_{\{ R_{s(t)} \leq (\frac{\beta}{2} + \epsilon)s(t)}\} \prod_{u \in N_{s(t)}} P^{X^u_{s(t)}} \Big( R_{t - s(t)} 
\leq \frac{\beta}{2}(t - s(t)) + \frac{\beta}{2}s(t) + y \Big) \to \exp \Big\{ - \mathrm{e}^{- \beta y} M_{\infty}\Big\} \text{.}
\]
Thus also
\[
\prod_{u \in N_{s(t)}} P^{X^u_{s(t)}} \Big( R_{t - s(t)} 
\leq \frac{\beta}{2}(t - s(t)) + \frac{\beta}{2}s(t) + y \Big) \to \exp \Big\{ - \mathrm{e}^{- \beta y} M_{\infty}\Big\} 
\]
and by bounded convergence
\[
E \Big[ \prod_{u \in N_{s(t)}} P^{X^u_{s(t)}} \Big( R_{t - s(t)} 
\leq \frac{\beta}{2}(t - s(t)) + \frac{\beta}{2}s(t) + y \Big) \Big] \to E \Big[ \exp \Big\{ - \mathrm{e}^{- \beta y} M_{\infty}\Big\} \Big] \text{.}
\]
For an arbitrary starting point of the branching process $x_0$ we have 
\begin{align*}
P^{x_0} \Big( R_t \leq \frac{\beta}{2}t + y \Big) = &P^{x_0} \Big( R_t \leq \frac{\beta}{2}t + y , \ T_0 > t\Big) + 
P^{x_0} \Big( R_t \leq \frac{\beta}{2}t + y , \ T_0 \leq t\Big)\\
= &P^{x_0} \Big( R_t \leq \frac{\beta}{2}t + y , \ T_0 > t\Big) + P^{x_0} \Big( \tilde{R}_{t - T_0} \leq \frac{\beta}{2}(t - T_0) + \frac{\beta}{2}T_0 + y , \ T_0 \leq t\Big)\\
\to \ &0 + E^{x_0} \Big( \exp \Big\{ - \mathrm{e}^{- \frac{\beta^2}{2}T_0 - \beta y} \tilde{M}_{\infty} \Big\} \Big) \\
= &E^{x_0} \Big( \exp \Big\{ - \mathrm{e}^{- \beta y} M_{\infty} \Big\} \Big) \text{,}
\end{align*}
where $T_0$ is the time the initial particle first hits the origin, $\tilde{R}_t = R_{t + T_0}$, $\tilde{M}_t = M_{t + T_0}$, $t \geq 0$.
\end{proof}

\bibliographystyle{acm}

\def\cprime{$'$}

\end{document}